\documentclass[11pt]{article}

\title{Localization for Linearly Edge Reinforced Random Walks}
\author{
  Omer Angel\thanks{Supported in part by NSERC} \and
  Nicholas Crawford\thanks{Supported in part at the Technion by a
    Landau fellowship} \and
  Gady Kozma\thanks{Supported in part by the Israel Science Foundation}
}
\date{March 2012}

\usepackage{amsmath,amsfonts,amssymb,amsthm,amstext}
\usepackage{enumerate,color}
\usepackage[colorlinks=true, linkcolor=black, citecolor=black,pdfborder={0 0 0}]{hyperref}
\usepackage{url}
\usepackage{mathrsfs}

\newif\iffinal
\finaltrue 

\usepackage{multido}
\usepackage{verbatim}
\usepackage{setspace}
\usepackage{nicefrac}

\ifpdf
 \usepackage[nameinlink]{cleveref} 
\else
 \usepackage{cleveref} 
\fi
\crefname{theorem}{Theorem}{Theorems}
\crefname{lemma}{Lemma}{Lemmas}
\crefname{section}{\S}{\S\S}
\crefname{equation}{}{}

\IfFileExists{lowers}{
 \ifpdf
 \else
  \crefname{theorem}{theorem}{theorems}
  \crefname{lemma}{lemma}{lemmas}
 \fi
}

\IfFileExists{myowntimes.sty}{\usepackage{myowntimes}\usepackage{mathrsfs}}
        {\usepackage{mathpazo}\usepackage{mathrsfs}}
\usepackage{calrsfs}

\newtheoremstyle{thm}{1.5ex}{1.5ex}{\itshape\rmfamily}{}
{\bfseries\rmfamily}{}{2ex}{}

\newtheoremstyle{def}{1.5ex}{1.5ex}{\slshape\rmfamily}{}
{\bfseries\rmfamily}{}{2ex}{}

\newtheoremstyle{rem}{1.3ex}{1.3ex}{\rmfamily}{}
{\itshape}
{} {1.5ex}{}

\theoremstyle{thm}
\newtheorem{theorem}{Theorem}
\newtheorem{lemma}[theorem]{Lemma}
\newtheorem{proposition}[theorem]{Proposition}

\newtheorem*{Main Theorem}{Main Theorem.}
\newtheorem{corollary}[theorem]{Corollary}
\newtheorem*{corollary*}{Corollary}
\newtheorem*{special theorem}{Lindeberg-Feller Theorem for Martingales}

\newtheorem*{definition*}{Definition}

\theoremstyle{rem}

\iffinal

\else

\fi

\newcommand{\eps}{\varepsilon}
\newcommand{\E}{\mathbb{E}}
\renewcommand{\P}{\mathbb{P}}

\newcommand{\R}{\mathbb R}
\newcommand{\Z}{\mathbb{Z}}

\DeclareMathOperator{\Bin}{Bin}
\DeclareMathOperator{\Exp}{Exp}
\DeclareMathOperator{\Geom}{Geom}
\DeclareMathOperator{\Bern}{Bern}
\DeclareMathOperator{\dist}{dist}
\DeclareMathOperator{\Vol}{Vol}

\renewcommand{\bar}[1]{\overline{#1}}
\newcommand{\eqd}{\stackrel{D}{=}}

\newcommand{\bJ}{\mathbf{J}}
\newcommand{\bW}{\mathbf{W}}
\newcommand{\CC}{\mathcal C}
\newcommand{\DD}{\mathcal D}

\makeatletter
\newenvironment{fullwidth}
{\par \setlength{\@totalleftmargin}{0pt}\setlength{\linewidth}{\hsize}%
  \list{}{\setlength{\leftmargin}{0pt}} \item\relax}
{\endlist}
\makeatother

\begin{document}

\maketitle

\begin{abstract}
  We prove that the linearly edge reinforced random walk (LRRW) on any
  graph with bounded degrees is recurrent for sufficiently small initial
  weights.  In contrast, we show that for non-amenable graphs the LRRW is
  transient for sufficiently large initial weights, thereby establishing a
  phase transition for the LRRW on non-amenable graphs.  While we rely on
  the description of the LRRW as a mixture of Markov chains, the proof does
  not use the magic formula.  We also derive analogous results for the
  vertex reinforced jump process.
\end{abstract}

\section{Introduction}
\label{sec:intro}

The {\bf linearly edge reinforced random walk} (LRRW) is a model of a
self-interacting (and hence non-Markovian) random walk, proposed by
Coppersmith and Diaconis, and defined as follows. Each edge $e$ of a graph
$G = (V,E)$ has an initial weight $a_e>0$. A starting vertex $v_0$ is
given. The walker starts at $v_0$. It examines the weights on the edges
around it, normalizes them to be probabilities. and then chooses an edge
with these probabilities. The weight of the edge traversed is then
increased by 1 (the edge is ``reinforced''). The process then repeats with
the new weights.

The process is called \emph{linearly} reinforced because the reinforcement
is linear in the number of steps the edge was crossed. Of course one can
imagine many other reinforcement schemes, and those have been studied (see
e.g.\ \cite{P07} for a survey). Linear reinforcement is special because the
resulting process is \emph{partially exchangeable}. This means that if
$\alpha$ and $\beta$ are two finite paths such that every edge is crossed
exactly the same number of times by $\alpha$ and $\beta$, then they have
the same probability (to be the beginning of an LRRW). Only linear
reinforcement has this property.

Diaconis \& Freedman \cite[Theorem 7]{DF} showed that a \emph{recurrent}
partially exchangeable process is a {\em mixture of Markov chains}. Today
the name {\em random walk in random environment} is more popular than
mixture of Markov chains, but they mean the same thing: that there is some
measure $\mu$ on the space of Markov chains (known as the ``mixing
measure'') such that the process first picks a Markov chain using $\mu$ and
then walks according to this Markov chain.  In particular, this result
applies to the LRRW whenever it is recurrent.  There ``recurrent'' means
that it returns to its starting vertex infinitely often.  We find this
result, despite its simple proof (it follows from de Finetti's theorem for
exchangeable processes, itself not a very difficult theorem) to be quite
deep. Even for LRRW on a graph with three vertices it gives non-trivial
information. For general exchangeable processes recurrence is necessary;
see \cite[Example 19c]{DF} for an example of a partially exchangeable
process which is not a mixture of Markov chains. For LRRW this cannot
happen, it is a mixture of Markov chains even when it is not recurrent (see
\cref{T:DF} below).

On finite graphs, the mixing measure $\mu$ has an explicit expression,
known fondly as the ``magic formula''. See \cite{MOR} for a survey of the
formula and the history of its discovery. During the last decade
significant effort was invested to understand the magic formula, with the
main target the recurrence of the process in two dimensions, a conjecture
dating back to the 80s (see e.g.\ \cite[\S 6]{Pem}).  Notably, Merkl and
Rolles \cite{MR2} showed, for any fixed $a$, that LRRW on certain "dilute"
two dimensional graphs is recurrent, though the amount of dilution needed
increases with $a$.  Their approach did not work for $\Z^2$, but required
stretching each edge of the lattice to a path of length $130$ (or more).
The proof uses the explicit form of the mixing measure, which turns out to
be amenable to entropy arguments.  These methods involve relative entropy
arguments which also lead to the Mermin-Wagner theorem \cite{MW}.  This
connection suggests that the methods should not apply in higher dimension.

An interesting variation on this theme is when each \emph{directed} edge
has a weight. When one crosses an edge one increases only the weight in the
direction one has crossed. This process is also partially exchangeable, and
is also described by a random walk in a random environment. On the plus
side, the environment is now i.i.d., unlike the magic formula which
introduces dependencies between the weights of different edges at all
distances. On the minus side, the Markov chain is not reversible, while the
Markov chains that come out of the undirected case are reversible. These
models are quite different. One of the key (expected) features of LRRW ---
the existence of a phase transition on $\Z^d$ for $d \geq 3$ from
recurrence to transience as $a$ varies --- is absent in the directed model
\cite{Sab,ES02}. In this paper we deal only with the reversible one.

Around 2007 it was noted that the magic formula is remarkably similar to
the formulae that appear in supersymmetric hyperbolic $\sigma$ models which
appear in the study of quantum disordered systems, see Disertori, Spencer
and Zirnbauer \cite{DSZ} and further see Efetov's book \cite{E} for an
account of the utility of supersymmetry in disordered systems. Very
recently, Sabot and Tarr\`es \cite{ST} managed to make this connection
explicit. Since recurrence at high temperature was known for the hyperbolic
$\sigma$ model \cite{DS}, this led to a proof that LRRW is recurrent in any
dimension, when $a$ is sufficiently small (high temperature in the $\sigma$
model corresponds to small $a$). We will return to Sabot and Tarr\`es's work
later in the introduction and describe it in more details. However, our
approach is very different and does not rely on the magic formula in any
way.

Our first result is a general recurrence result:

\begin{theorem}\label{T:recurrence}
  For any $K$ there exists some $a_0>0$ such that if $G$ is a graph with
  all degrees bounded by $K$, then the linearly edge reinforced random walk
  on $G$ with initial weights $a\in(0,a_0)$ is a.s.\ positive recurrent.
\end{theorem}

Positive recurrence here means that the walk asymptotically spends a
positive fraction of time at any given vertex, and has a stationary
distribution.  In fact, the LRRW is equivalent to a random walk in a
certain reversible, dependent random environment (RWRE) as discussed below.
We show that this random environment is a.s.\ positively recurrent.
We formulate and prove this theorem for constant initial weights. However,
our proof also works if the initial weights are unequal as long as for each
edge $e$, the initial weight $a_e$ is at most $a_0$. With minor
modifications, our argument can be adapted to the case of random
independent $a_e$'s with sufficient control of their moments. See the
discussion after the proof of \cref{L:AUB} for details.

Let us stress again that we do not use the explicit form of the magic
formula in the proof. We do use that the process has a reversible RWRE
description, but we do not need any details of the measure. The main
results in \cite{ST} are formulated for the graphs $\Z^d$ but Remark $5$ of
that paper explains how to extend their proof to all bounded degree
graphs. Further, even though \cite{ST} claims only recurrence, positive
recurrence follows from their methods. Thus the main innovation here is the
proof technique.

It goes back to \cite{DF} that in the recurrent case the weights are
uniquely defined after normalizing, say by $W_{v_0}=1$ (see there also an
example of a transient partially exchangeable process where the weights are
not uniquely defined).  Hence with \cref{T:recurrence} the weights are well
defined and we may investigate them.  Our next result is that the weights
decay exponentially in the graph distance from the starting point.  Denote
graph distance by $\dist(\cdot,\cdot)$.  Also, let $\dist(e, v_0)$ denote
the minimal number of edges in any path from $v_0$ to either endpoint of
$e$.

\begin{theorem}\label{T:ExpDecay}
  Let $G$ be a graph with all degrees bounded by $K$ and let $s\in
  (0,\nicefrac14)$. Then there exists $a_0 = a_0(s,K) > 0$ such that for
  all $a\in(0,a_0)$,
  if $e_1$ is the first edge crossed by $X$,
  \begin{equation}\label{E:UB2}
    \E\left( W_e \right)^s
    \leq \E\left(\frac{W_e}{W_{e_1}} \right)^s
    \leq 2K \Big( C(s,K) \sqrt{a} \Big)^{\dist(e,v_0)}.
  \end{equation}
\end{theorem}

Note that $a_0$ does not depend on the graph except through the maximal
degree. The factor $2K$ on the right-hand side is only relevant, of course,
when $\dist(e,v_0)=0$, otherwise it can be incorporated into the constant
inside the power. 

The parameter $s$ deserves some explanation.  It is interesting to note
that despite the exponential spatial decay, each $W_i$ does not have
arbitrary moments.  Examine, for example, the graph with three vertices and
two edges. The case when the initial edge weights are $2$ and the process
starts from the center vertex is particularly famous as it is equivalent to
the standard P\'olya urn.  In this case the weights are distributed
uniformly on $[0,1]$. This means that the ratio of the weights of the two
edges does not even have a first moment. Of course, this is not quite the
quantity we are interested in, as we are interested in the ratio $W_e/W_f$
where $W_f$ is the first edge crossed. This, though, is the same as an LRRW
with initial weights $a$, $a+1$ and starting from the side. Applying the
magic formula can show that the ratio has $1+a/2$ moments, but not more.
It is also known directly that in this generalized P\'olya urn the weights
have a Beta distribution (we will not do these calculations here, but they
are straightforward).  Hence care is needed with the moments of these
ratios.

Our methods can be easily improved to give a similar bound for $s<\frac13$,
with $\sqrt{a}$ replaced by a suitably smaller power. See the discussion
after \cref{L:AUB} for details. Our proof can probably be modified to give
$\nicefrac{1}{2}$ a moment, and depending on $a$, a bit more. Going beyond
that seems to require a new idea.

The most interesting part of the proof of \cref{T:ExpDecay,T:recurrence} is
to show \cref{T:ExpDecay} given that the process is already known to be
recurrent, and we will present this proof first in \cref{sec:proof}.
\Cref{T:recurrence} then follows easily by approximating the graph with
finite subgraphs where the LRRW is of course recurrent. This is done in
\cref{sec:recurr}.

\subsection{Transience}

An exciting aspect of LRRW is that on some graphs it undergoes a phase
transition in the parameter $a$.  LRRW on trees was analyzed by
Pemantle \cite{Pem}. He showed that there is a critical $a_c$ such that for
initial weights $a<a_c$, the process is positively recurrent, while for
$a>a_c$ it is transient.  We are not aware of another example where a phase
transition was proved, nor even of another case where the process was shown
to be transient for any initial weights.  The proof of Pemantle relies
critically on the tree structure, as in that case, when you know that you
have exited a vertex through a certain edge, you know that you will return
through the very same edge, if you return at all.  This decomposes the
process into a series of i.i.d.\ P\'olya urns, one for
each vertex.  Clearly this kind of analysis can only work on a tree.

The next result will show the existence of a transient regime in the case of
non-amenable graphs. Recall that a graph $G$ is {\bf non-amenable} if for
some $\iota>0$ and any finite set $A \subset G$,
\[
|\partial A| \geq \iota |A|
\]
where $\partial A$ is the external vertex boundary of $A$ i.e.\ $\partial A
= \{x: \dist(x,A) = 1\}$. The largest constant $\iota$ for which this holds
is called the {\bf Cheeger constant} of the graph $G$.

\begin{theorem}\label{T:transience}
  For any $K,c_0>0$ there exists $a_0$ so that the following holds.  Let
  $G$ be a graph with Cheeger constant $\iota \geq c_0$ and degree bound
  $K$.  Then for $a>a_0$ the LRRW on $G$ with all initial weights $a$ on
  all edges is transient.
\end{theorem}

\Cref{T:transience} will be proved in \cref{sec:nonamen}.  As with
\cref{T:recurrence}, our proof works with non-equal initial weights,
provided that $a_e>a_0$ for all $e$.  It is tempting to push for stronger
results by considering graphs $G$ with intermediate expansion properties
such that the simple random walk on $G$ is transient. To put this in
perspective, let us give some examples where LRRW has no transient regime.

\medskip
\noindent {\bf 1.} The \emph{canopy graph} is the graph $\mathbb Z^+$ with
a finite binary tree of depth $n$ attached at vertex $n$. It is often
poetically described as ``an infinite binary tree seen from the leaves''.
Since the process must leave each finite tree eventually, the process on
the ``backbone'' is identical to an LRRW on $\mathbb Z^+$, which is
recurrent for all $a$ (say from \cite{D90} or from the techniques of
\cite{Pem}).

\medskip \noindent {\bf 2.} Let $T$ be the infinite binary tree.  Replace
each edge on the $n^{\textrm{th}}$ level by a path of length $n^2$.  The
random walk on the resulting graph is transient (by a simple resistance
calculation, see e.g.~\cite{Doyle}). Nevertheless, LRRW is recurrent for
any value of $a$. This is because LRRW on $\mathbb Z^+$ has the expected
weights decaying exponentially (again from the techniques of \cite{Pem})
and this decay wins the fight with the exponential growth of the levels.

\medskip \noindent {\bf 3.} These two example can be combined (a stretched
binary tree with finite decorations) to give a transient graph with
exponential growth on which LRRW is recurrent for any $a$.

\medskip\noindent We will not give more details on any of these examples as
that will take us off-topic, but they are all more-or-less easy to do.  The
proof of \cref{T:transience} again uses that the process has a dual
representation as both an self-interacting random walk and as an RWRE.
This might be a good point to reiterate that the LRRW is a mixture of Markov
chains even in the transient case (the counterexample of Diaconis and
Freedman is for a partially exchangeable process, but that process is not a
LRRW).  This was proved by Merkl and Rolles \cite{MR4}, who developed a
tightness result for this purposes which will also be useful for us.  Let
us state a version of their result in the form we will need here along with
the original result of Diaconis and Freedman:

\begin{theorem}[\cite{DF, MR4}]\label{T:DF}
  Fix a finite graph $G=(V,E)$ and initial weights $\mathbf a = (a_e)_{e\in
    E}$ with $a_e>0$. Then $X$ is a mixture of Markov chains in the
  following sense: There exists a unique probability measure $\mu$ on
  $\Omega$ so that
  \[
  \P = \int \P^{\mathbf W} d\mu(\mathbf W)
  \]
  is an identity of measures on infinite paths in $G$. All $W(e)>0$ for all $e$.
  
  Moreover, the $W_v$'s form a tight family of random variables: if we set
  $W_v= \sum_{e \ni v} W_e$ and $a_v =\sum_{e \ni v} a_e$, there are
  contestants $c_1,c_2$ depending only on $a_v, a_e$ so that
  \[
  \mu\left( W_e/W_v \leq \eps \right) \leq c_1 \eps^{a_e/2} \quad
  \text{and} \quad \mu\left( W_e/W_v \geq 1-\eps \right) \leq c_2
  \eps^{(a_v-a_e)/2}.
  \]
\end{theorem}

As noted, we mainly need the existence of such a representation of the
LRRW, as well as the tightness of the $W$'s, but not the explicit bounds.

It is natural to conjecture that in fact a phase transition exists on
$\Z^d$ for all $d\ge 3$, analogously to the phase transition for the
supersymmetric $\sigma$ model. What happens in $d=2$? We believe it is
recurrent for all $a$, but dare not guess whether it enjoys a
Kosterlitz-Thouless transition or not.

\subsection{Back to Sabot and Tarr\`es}

The main object of study for Sabot and Tarr\`es \cite{ST} are vertex
reinforced jump processes (VRJP). Unlike LRRW, this is defined in
continuous time process with reinforcement acting through local times of
vertices. One begins with a positive function $\bJ=(J_e)_{e \in E}$ on the
edges of $G$. These are the initial rates for the VRJP process $Y_t$, and
are analogous to the initial weights $a$ of LRRW. Let $(L_x(t))_{x\in V}$
be the local times for $Y$ at time $t$ and vertex $x$. If $Y_t=x$ then $Y$
jumps to a neighbor $y$ with rate $J_{xy}(1+L_y(t))$. See \cite{ST} for
history and additional reference for this process.

The VRJP shares a key property with the LRRW: a certain form of partial
exchangeability after applying a certain time change. This suggests that it
too has a RWRE description (Diaconis and Freedman \cite{DF} only consider
discrete time processes, but their ideas should carry over). Such a RWRE
description exists was found by Sabot and Tarr\`es by other methods. The
existence (though not the formula) for such a form is fundamental for our
proof. Their main results are the following: First, the law of LRRW with
initial weight $a$ is identical to the time-discretization of $Y_t$ when
$\bJ$ is i.i.d.\ with marginal distribution $\Gamma(a, 1)$. Secondly, after
a time change, $Y_t$ is identical to a mixture of continuous-time
reversible Markov chains.  Moreover, the mixing measure is {\em exactly}
the marginal of the supersymmetric hyperbolic $\sigma$ models studied in
\cite{DS, DSZ}. This is analogous to the magic formula for the
LRRW. Finally, as already mentioned, this allowed them to harness the
techniques of \cite{DS, DSZ} to prove that LRRW is recurrent for small $a$
in all dimensions.

Since the VRJP has both a dynamic and an RWRE representation, our methods
apply to this model too. Thus we show:

\begin{theorem}\label{T:VRJPrecurr}
  Let $G$ be a fixed graph with degree bound $K$.  Let $\bJ= (J_{e})_{e \in
    E}$ be a family of independent initial rates with
  \[
  \E J^{1/5} < c(K),
  \]
  where $c(K)$ is a constant depending only on $K$.  Then (a.s.\ with
  respect to $\bJ$), $Y_t$ is recurrent.
\end{theorem}

In particular this holds for any fixed, sufficiently small $\bJ$. We
formulate this result for random $\bJ$ because of the relation to LRRW
explained above --- we could have also proved the LRRW results for random
$a$ but we do not see a clear motivation to do so. We will prove this in
\cref{S:VRJP}, where we will also give more precise information about the
dependency of $c$ on $K$.

Next we wish to present a result on exponential decay of the weights, an
analogue of \cref{T:ExpDecay}. To make the statement more meaningful, let
us describe the RWRE (which, we remind, is not $Y_t$ but a time change of
it). We are given a random function $W$ on the \emph{vertices} of our graph
$G$. The process is then a standard continuous-time random walk which moves
from vertex $i$ to vertex $j$ with rate $\frac12 J_{ij}W_j/W_i$. The result
is now:

\begin{theorem}\label{T:VRJP}
  If the $J_e$ are independent, with $\E J_e^{1/5} < c(K)$ then for a.e.\
  $\bJ$, the infinite volume mixing measure exists and under the joint
  measure, for any vertex $v \in G$,
  \[
  \E W_v^{1/5} < 2K^{-4\dist(v_0,v)}.
  \]
\end{theorem}

In particular, this implies that the time-discretization of $Y$ is
positively recurrent, and not just recurrent.  For any $s<\nicefrac15$,
bounds on $\E J^s$ can yield the an estimate on $\E W_v^s$ and
\cref{T:VRJPrecurr}.

It is interesting to note that the proof of \cref{T:VRJPrecurr,T:VRJP} are
simpler than that of \cref{T:recurrence,T:ExpDecay}, even though the
techniques were developed to tackle LRRW. The reader will note that for
VRJP, each edge can be handled without interference from adjoining edges on
the same vertex, halving the length of the proof. Is there some inherent
reason? Is VRJP (or the supersymmetric $\sigma$ model) more basic in some
sense? We are not sure.

\subsection{Notations}

In this paper, $G$ will always denote a graph with bounded degrees, and $K$ a bound on these degrees. The set of edges of $G$ will be denoted by $E$. $a_e$ will always denote the initial weights, and when the notation $a$ is used, it is implicitly assumed that $a_e=a$ for all edges $e\in E$.

Let us define the LRRW again, this time using the notations we will use for
the proofs. Suppose we have constructed the first $k$ steps of the walk,
$x_0,\dotsc,x_k$. For each edge $e \in E$, let
\[
N_k(e) = |\{j < k: e=\langle X_j, X_{j+1} \rangle\}|
\]
be the number of times that the {\em undirected} edge $e$ has been
traversed up to time $k$. Then each edge $e$ incident on $X_k$ is used for
the next step with probability proportional to $a + N_k(e)$,  that is,
If $X_k=v$ then 
\[
\P(X_{k+1} = w | X_0,\dots,X_k) = \frac{a+N_k((v,w))}{d_v a + N_k(v)} 
\mathbf 1\{v \sim w\},
\]
where $d_v$ is the degree of $v$; where $N_k(v)$ denotes the sum of
$N_k(e)$ over edges incident to $v$; and where $\sim$ is the neighborhood
relation i.e.\ $v\sim w\iff \langle v,w\rangle\in E$.  It is crucial to
remember that $N_k$ counts traversals of the undirected edges. We stress
this now because at some points in the proof we will count oriented traversals
of certain edges.

While all graphs we use are undirected, it is sometimes convenient to think
of edges as being directed. Each edge $e$ has a head $e^-$ and tail
$e^+$. The reverse of the edge is denoted $e^{-1}$. A path of length $n$ in
$G$ may then be defined as a sequence of edges $e_1,\dots,e_n$ such that
$e_i^+ = e_{i+1}^-$ for all $i$. Vice versa, if $v$ and $w$ are two
vertices, $\langle v,w\rangle$ will denote the edge whose two vertices are
$v$ and $w$.

By $\bW$ we will denote a function from $E$ to $[0,\infty)$ (``the
weights''). We will denote $W_e$ instead of $\bW(e)$ and for a vertex $v$
we will denote
\[
W_v:=\sum_{e\ni v}W_e.
\]
The space of all such $\bW$ will be denoted by $\Omega$ and measures on it
typically by $\mu$.  The $\mu$ which describes our process (whether unique
or not) will be called ``the mixing measure''.

Given $\bW$ we define a Markov process with law $\P^\bW$ on the vertices of
$G$ as follows.  The probability to transition from $v$ to $w$, denoted
$\P^\bW(v,w)$ is $W_{\langle v,w\rangle}/W_v$. For a given $\mu$ on
$\Omega$, the RWRE corresponding to $\mu$ is a process on the vertices of
$G$ given by
\[
\P(X_0=v_0,\dotsc,X_n=v_n)=\int \prod_{i=1}^{n}\mathbb P^\bW(X_{i-1},X_i)\,d\mu(\bW) .
\]
This process will always be denoted by $X$.

\begin{definition*}
  A process is recurrent if it returns to every vertex infinitely many
  times.
\end{definition*}

This is the most convenient definition of recurrence for us. It is formally
different from the definition of \cite{DF} we quoted earlier, but by the
results of \cite{MR4} quoted above they are in fact equivalent for LRRW.

We also use the following standard notation: for two vertices $v$ and $w$
we define $\dist(v,w)$ as the length of the shortest path connecting them
(or 0 if $v=w$). For an edge $e$ we define
$\dist(v,e)=\min\{\dist(v,e^-),\dist(v,e^+)\}$. For a set of vertices $A$
we define $\partial A=\{x:\dist(x,A)=1\}$ i.e.\ the external vertex
boundary. By $X\eqd Y$ we denote that the variables $X$ and $Y$ have the
same distribution. $\Bern(p)$ will denote a Bernoulli variable with
probability $p$ to take the value 1 and $1-p$ to take the value 0 and
$\Exp(\lambda)$ will denote an exponential random variable with expectation
$1/\lambda$. We will denote constants whose precise value is not so
important by $c$ and $C$, where $c$ will be used for constants sufficiently
small, and $C$ for constants sufficiently large. The value of $c$ and $C$
(even if they are not absolute but depend on parameters, such as $C(K)$)
might change from one place to another. By $x\asymp y$ we denote $cx\le
y\le Cx$.

\subsection{Acknowledgements}

We wish to thank Tom Spencer for suggesting to us that a result such as
\cref{T:ExpDecay} might be true. Part of this work was performed during the
XV Brazilian school of probability in Mambucaba. We wish to thank the
organizers of the school for their hospitality.

\section{Proof of \texorpdfstring{\cref{T:ExpDecay}}{Theorem \ref{T:ExpDecay}}}\label{sec:proof}

In this section we give the most interesting part of the proof of
\cref{T:ExpDecay}, showing exponential decay {\em assuming a priori that
  the process is recurrent}. We give an upper bound which only depends on
the maximal degree $K$. In the next section use apply this result to a
sequence of finite approximations to $G$ to prove recurrence for the whole
graph and complete the proof.

Before we begin, we need to introduce a few notations.  For any edge $e =
(e^-,e^+)$ that is traversed by the walk, let $e'$ be the edge through
which $e^-$ is first reached. In particular, $e'$ is traversed before
$e$. If $e$ is traversed before its inverse $e^{-1}$, then $e'$ is distinct
from $e$ as an undirected edge. Iterating this construction to get
$e'',e'''$, etc.\ yields a path $\gamma = \{\dots, e'',e',e\}$ starting
with the first edge used from $v_0$, and terminating with $e$. We call
$\gamma$ the {\bf path of domination} of $e$. This path is either a simple
path, or a simple loop in the case that $e^+$ is the starting vertex
$v_0$. In the former case $\gamma$ is the backwards loop erasure of 
the LRRW.  All edges in the path are traversed before their corresponding
inverses. Let $\DD_\gamma$ be the event that the deterministic path
$\gamma$ is the path of domination corresponding the final edge of
$\gamma$.

For an edge $e$ with $e'\neq e^{-1}$, let $Q(e)$ be an estimate for
$W_e/W_{e'}$ defined as follows: If $e$ is crossed before
$f:=(e')^{-1}$ then set $M_e$ to be the number of times $e$ is crossed
before $f$, and set $M_f=1$.  If $f$ is crossed before $e$ then set $M_f$
to be the number of times $f$ is crossed before $e$, and set $M_e=1$. In
both cases we count crossing \emph{as directed edges}.  In other words, we
only count crossings that start from $e^-$, the common vertex of $e$ and
$f$ (in which case the walker chooses between them).  Then
\[
  Q(e) := \frac{M_e}{M_f}
\]
is our estimate for $W_e/W_{e'}$.  Thus to find $Q(e)$ we wait
until the LRRW has left $e^-$ along both $e$ and $f$, and take the ratio of
the number of departures along $e$ and along $f$ at that time. Note again that we
do not include transitions along $e^{-1}$ or $f^{-1}=e'$, so that by
definition one of the two numbers $M_e$ and $M_f$ must be $1$.

With $Q$ defined we can start the calculation. Recall that $e_1$ is the
first edge crossed by the walk. Suppose $x$ is some edge of $G$ and we want
to estimate $W_x/W_{e_1}$. We fix $x$ for the reminder of the proof. Let
$\Gamma=\Gamma_x$ denote the set of possible values for the path of
domination i.e.\ the set of simple paths or loops whose first edge is one
of the edges coming out of $v_0$ and whose last edge is either $x$ or
$x^{-1}$ (depending on which is crossed first).

Split the probability space according to the value of the path of domination:
\[
\E \left[ \left( \frac{W_x}{W_{e_1}} \right)^s \right]
= \sum_{\gamma \in \Gamma} \E \left[ \left( \frac{W_x}{W_{e_1}} \right)^s
  \mathbf 1\{\DD_{\gamma}\} \right].
\]
Naturally, under $\DD_\gamma$, $e_1$ must be the first edge of $\gamma$. We
remind the reader that we assume {\em a priori} that our process is
recurrent. This has two implications: first, $x$ will a.s.\ be visited
eventually, and so the path of domination is well-defined. Second, the
weights are unique, so $W_x/W_{e_1}$ is a well-defined variable on our
probability space.

Given the weights $W$, let $R$ be the actual ratios along the path of
domination, so for $e\in\gamma\setminus\{e_1\}$,
\[
R(e)= \frac{W_e}{W_{e'}}.
\]
On the event $\DD_{\gamma}$ for $\gamma$ fixed, we may telescope
$W_x/W_{e_1}$ via 
\[
\frac{W_x}{W_{e_1}}
= \prod_{e \in \gamma\setminus e_1} R(e)
= \prod_{e \in \gamma\setminus e_1}\frac{R(e)}{Q(e)}
  \prod_{e \in \gamma\setminus e_1} Q(e).
\]
An application of the Cauchy-Schwarz inequality then gives
\begin{equation}
\label{E:CS1}
\E\left[\left(\frac{W_x}{W_{e_1}}\right)^s \mathbf 1\{\DD_{\gamma}\}\right]
\leq 
\E \left[ \prod_{e \in \gamma\setminus e_1} \left( \frac{R(e)}{Q(e)}
  \right)^{2s} \mathbf 1 \{\DD_{\gamma}\}\right]^{\frac12}  
\E \left[ \prod_{e \in \gamma\setminus e_1} Q(e)^{2s} \mathbf 1
  \{\DD_{\gamma}\} \right]^{\frac 12}.
\end{equation}
\cref{T:ExpDecay} then essentially boils down to the following two lemmas.

\begin{lemma}\label{L:RUB}
  For any graph $G$, any starting vertex $v_0$ and any $a \in (0, \infty)$
  such that the LRRW on $G$ starting from $v_0$ with initial weights $a$ is
  recurrent, for any edge $x \in E$, any $\gamma\in\Gamma_x$ and any $s \in
  (0, 1)$,
  \[
  \E \left[ \prod_{e \in \gamma\setminus e_1} \left( \frac{R(e)}{Q(e)}
    \right)^s \mathbf 1 \{\DD_{\gamma}\} \right]
  \leq C(s)^{|\gamma|-1}
  \]
  where $C(s)$ is some constant that depends on $s$ but not on $G$, $v_0$,
  $a$ or anything else.
\end{lemma}

\begin{lemma}\label{L:AUB}
  For any graph $G$ with degrees bounded by $K$, any starting vertex $v_0$
  and any $a \in (0, \infty)$ such that LRRW on $G$ starting from $v_0$
  with initial weights $a$ is recurrent, for any edge $x \in G$, any
  $\gamma\in\Gamma_x$ and any $s \in (0,\nicefrac12)$,
  \[
  \E\bigg[ \prod_{e\in\gamma\setminus e_1} Q(e)^s \mathbf 1\{\DD_\gamma\}
  \bigg] \leq \bigg[C(s,K) a \bigg]^{|\gamma|-1},
  \]
  where $C(s,K)$ is a constant depending only on $s,K$.
\end{lemma}

In these two lemmas, it is not difficult to make $C(s)$ and $C(s,K)$
explicit. following the proof gives $C(s) = O(\frac{1}{1-s})$ and $C(s,K) =
O(\frac{K}{1-2s} + \frac1s)$.  However, there seems to be little reason to
be interested in the $s$-dependency. We will apply the lemmas with some
fixed $s$, say $\nicefrac14$.  We do not have particularly strong feelings
about the $K$-dependency either.  It is worth noting that \cref{L:RUB} is
proved by using the random environment point of view on the LRRW, while
\cref{L:AUB} is proved by considering the reinforcements.  Thus both views
are central to our proof.  We note that \cref{T:ExpDecay} can be extended
to $s<\nicefrac13$ by simply using H\"older's inequality instead of
Cauchy-Schwartz.  Relaxing the limit on $s$ in \cref{L:RUB} may allow any
$s<\nicefrac12$, but that is the limit of our approach, since the
$\nicefrac12$ in \cref{L:AUB} is best possible.

\begin{proof}[Proof of \cref{L:RUB}]
  For this lemma the RWRE point of view is used, as it must, since the
  weights $W$ appear in the statement, via $R$. Our first step is to throw
  away the event $\mathbf 1\{\DD_\gamma\}$ i.e.\ to write
  \begin{equation}\label{eq:throwDgam}
    \E \left[ \prod_{e \in \gamma\setminus e_1} \left( \frac{R(e)}{Q(e)}
      \right)^s \mathbf 1 \{\DD_{\gamma}\} \right]
    \leq 
    \E \left[ \prod_{e \in \gamma\setminus e_1} \left(
        \frac{R_\gamma(e)}{Q_\gamma(e)} 
      \right)^s \right]
  \end{equation}
  where the terms on the right-hand side are as follows: $R_\gamma(e)$ is
  the ratio between $W_e$ and $W_f$ where $f$ is the predecessor of $e$ in
  $\gamma$; and $Q_\gamma(e)$ is defined by following the process until
  both $e$ and $f$ are crossed at least once from $e^-$ and then define
  $M_e$, $M_f$ and $Q$ according to these crossings. Clearly, under
  $\DD_\gamma$ both definitions are the same so \eqref{eq:throwDgam} is
  justified.

  This step seems rather wasteful, as heuristically one can expect to lose
  a factor of $K^{|\gamma|}$ from simply ignoring a condition like
  $\DD_\gamma$. But because our eventual result (\cref{T:ExpDecay}) has a
  $C(s,K)^{|\gamma|}$ term, this will not matter.  Since $\gamma$ is fixed,
  from this point until the end of the proof of the lemma we will denote
  $R=R_\gamma$ and $Q=Q_\gamma$.

  At this point, and until the end of the lemma, we fix one realization of the
  weights $\bW$ and condition on it being chosen. This conditioning makes
  the $R(e)$ just numbers, while the $Q(e)$ become independent. Indeed,
  given $W$, the random walk in the random environment can be constructed
  by associating with each vertex $v$ a sequence $Z^v_n$ of i.i.d.\ edges
  incident to $v$ with law $W_e/\sum_{e\ni v} W_e$. If the walk is
  recurrent, then $v$ is reached infinitely often, and the entire sequence
  is used in the construction of the walk. If we fix an edge $e$ and let
  $f=e'^{-1}$, then $M_e$ (resp.\ $M_f$) is the number of appearances of
  $e$ (resp.\ $f$) in the sequence $\{Z^v\}$ up to the first time that $e$
  and $f$ have both appeared. As a consequence, since the sequences for
  different vertices are independent, we get that conditioned on the
  environment $\bW$, the estimates $Q(e)$ for $W_e/W_f$ for pairs
  incident to different vertices are all independent.

  Thus to prove our lemma it suffices to show that for any two edges $e,f$
  leaving some vertex $v$, with $M_e$ and $M_f$ defined as above we have
  \[
  \E \left[\left( \frac{W_e}{W_f} \frac{M_f}{M_e} \right)^s \bigg| \mathbf
    W\right] \leq C(s).
  \]
  We now show that this holds uniformly in the environment.

  First, observe that entries other than $e,f$ in the sequence of i.i.d.\
  edges at $v$ have no effect on the law of $M_e$ and $M_f$, so we may
  assume w.l.o.g.\ that $e$ and $f$ are the only edges coming out of
  $e^-$. Denote the probability of $e$ by $p$ and of $f$ by $q=1-p$ (for
  some $p\in(0,1)$). Now we have for $n\geq1$, that the probability that
  $e$ appears $n$ times before $f$ is $p^n q$, and similarly for $f$ before
  $e$ with the roles of $p$ and $q$ reversed. Thus
  \[
  \E \left[\left( \frac{W_e}{W_f} \frac{M_f}{M_e} \right)^s
    \bigg| \mathbf W\right]
  = \underbrace{\left(\frac{p}{q}\right)^s}_{W_e/W_f} 
  \bigg[ \underbrace{\sum_{k \geq 1} k^s p q^k}_{f \text{ first}}
  + \underbrace{\sum_{k \geq 1} k^{-s} q p^k}_{e \text{ first}} \bigg] .
  \]
  It is a straightforward calculation to see that this is bounded for
  $|s|<1$. The first term is the $s$-moment of a $\Geom(p)$ random variable,
  which is of order $p^{-s} q$, and with the pre-factor comes to
  $q^{1-s}$. The second term is the $(-s)$-moment of a $\Geom(q)$ random
  variable, which is of order $p q^s$, and with the pre-factor gives
  $p^{s+1}$. Thus
  \[
  \E \left[\left( \frac{W_e}{W_f} \frac{M_f}{M_e} \right)^s
    \bigg| \mathbf W\right]
  \asymp q^{1-s} + p^{1+s} \leq C(s)         
  \]
(recall that we assumed $s\in(0,1)$). This finishes the lemma.
\end{proof}

Now we move on to the proof of \cref{L:AUB}.

\begin{proof}[Proof of \cref{L:AUB}]
  Fix a path $\gamma$. We shall construct a coupling of the LRRW together
  with a collection of i.i.d.\ random variables $\bar{Q}(e)$, associated
  with the edges of $\gamma$ (except $e_1$) such that on the event
  $\DD_\gamma$, for every edge $e\in\gamma\setminus e_1$ we have $Q(e)\leq
  \bar{Q}(e)$, and such that for $s<\nicefrac12$,
  \[
  \E \bar{Q}(e)^s \leq C(s,K) a.
  \]
  The claim would follows immediately, because
  \begin{align*}
    \E \prod Q(e)^s\mathbf 1\{\DD_\gamma\}
    &\le \E \prod \bar Q(e)^s\mathbf 1\{\DD_\gamma\} \le
    \E\prod \bar Q(e)^s=\prod \E \bar Q(e)^s\\
    &\le\prod C(s,K)a=\big(C(s,K)a\big)^{|\gamma|-1}.
  \end{align*}
  The remarks in the previous proof about ``waste'' are just as applicable
  here, since we also, in the second inequality, threw away the event
  $\DD_\gamma$.  Note that we cannot start by eliminating the restriction,
  since we only prove $Q\leq\bar{Q}$ on the event $\DD_\gamma$.

  Let us first describe the random variables $\bar{Q}$. Estimating their
  moments is then a straightforward exercise.  Next we will construct the
  coupling, and finally we shall verify that $Q(e)\leq \bar{Q}(e)$.  For an
  edge $e=(e^-,e^+)$ of $\gamma$, we construct two sequences of Bernoulli
  random variables (both implicitly depending on $e$).  For $j\geq0$,
  consider Bernoulli random variables
  \begin{align*}
    Y_j &= \Bern\left(\frac{a}{j+1+2a}\right),  &
    Y'_j &= \Bern\left(\frac{1+a}{2j+1+Ka}\right).
  \end{align*}
  where $\Bern(p)$ is a random variable that takes the value 1 with
  probability $p$ and 0 with probability $1-p$.  All $Y$ and $Y'$ variables
  are independent of each other and of those associated with other edges in
  $\gamma$. In the context of the event $\DD_\gamma$, we think of $Y_0'$ as
  the event that decides which of $e$ and $f$ is crossed first. For $j\ge
  1$, think about $Y_j$ as telling us whether on the $j^\textrm{th}$ visit
  to $e^-$  we depart along $e$ and $Y'_j$ telling us whether we depart
  along $f=e'^{-1}$. This leads to the definition
  \[
  \bar{Q} = \bar{M}_e / \bar{M}_f,
  \]
  where
  \begin{align*}
    \bar{M}_e &= \min\{j\geq 1 : Y'_j=1\}
    \qquad \text{and} \qquad      \bar{M}_f = 1, &
    \text{if $Y'_0=0$}, \\
    \bar{M}_f &= \min\{j\geq 1 : Y_j=1\}
    \qquad \text{and} \qquad     \bar{M}_e = 1, &
    \text{if $Y'_0=1$}. \\
  \end{align*}
\noindent{\bf Moment estimation.} To estimate $\E \bar{Q}^s$ we note
  \[
  \P(Y'_0=0, \bar{M}_e = n)
  = \frac{(K-1)a}{1+Ka}\frac{1+a}{2n+1+Ka} \prod_{j=1}^{n-1}
  \left(1-\frac{1+a}{2j+1+Ka} \right).
  \]
  The first two terms we estimate by
  \[
  \frac{(K-1)a}{1+Ka}\frac{1+a}{2n+1+Ka}\le
  \frac{Ka}{1+Ka}\frac{1+Ka}{2n}=\frac{Ka}{2n}
  \]
  while for the product we note that for any $a>0$,
  \[
  \frac{1+a}{2j+1+Ka}\ge \min\left\{\frac{1}{2j+1},\frac{1}{K}\right\}.
  \]
  Putting these together we get 
    \begin{align*}
      \P(Y'_0=0, \bar{M}_e = n)&\le \frac{Ka}{2n}
    \prod_{j=1}^{n-1}
    \exp\left(-\frac{1}{2j}+O(j^{-2})\right)
    = \frac{Ka}{2n} \exp\left(-\frac{1}{2}\log(n)+O(1)\right) \\
    &\leq C(K) a n^{-3/2}.
  \end{align*}
  Thus for $s<\nicefrac12$,
  \begin{align*}
    \E \Big[ \bar{Q}(e)^s \mathbf 1\{Y'_0=0\} \Big]
    &\leq \sum_{n\geq 1} n^s \P(Y'_0=0, \bar{M}_e = n) \\
    &\leq \sum_{n\geq 1} C(K) a n^{s-3/2} 
    \leq C(s,K) a.
  \end{align*}
  (This is the main place where the assumption $s<\nicefrac12$ is used.)
  
  For the case $Y_0'=1$ we write
  \[
  \P(Y'_0=1, \bar{M}(f) = n) \leq \P(Y_n=1) \leq \frac{a}{n},
  \]
  and so
  \[
  \E \Big[ \bar{Q}(e)^s \mathbf 1\{Y'_0=1\} \Big] = \sum_{n\geq 1} a n^{-1-s} <
  C(s) a.
  \]
  Together we find $\E \bar{Q}(e)^s \leq C(s,K) a$ as claimed.
  
  \paragraph*{The coupling.}
  Here we use the linear reinforcement point of view of the walk.  We
  consider the Bernoulli variables as already given, and construct the LRRW
  as a function of them (and some additional randomness).  Suppose we have
  already constructed the first $t$ steps of the LRRW, are at some vertex
  $v$, and need to select the next edge of the walk. There are several
  cases to consider.

  \begin{description}
  \item[Case 0.] $v\notin\gamma$: In this case we choose the next edge
    according to the reinforced edge weights, independently of all the $Y$
    and $Y'$ variables.
  
  \item[Case 1.] $v\in\gamma$ and the LRRW so far is not consistent with
    $\DD_\gamma$: We may safely disregard the $Y$ variables, as nothing is
    claimed in this case. This case occurs if for some edge $e\in\gamma$ is
    traversed only after its inverse is, or if the first arrival to $e^-$
    was not through the edge preceding $e$ in $\gamma$.

  \item[Case 2.] $v\in\gamma$, and $Q$ is already determined: If $v=e^-$
    for $e\in\gamma$, and both $e$ and $(e')^{-1}$ have both already been
    traversed, then $Q(e)$ is determined by the path so far. Again, we
    disregard the $Y$ variables.

    \begin{fullwidth}
      For the remaining cases, we may assume that $\DD_\gamma$ is
      consistent with the LRRW path so far, and that $v=e^-$ for
      $e\in\gamma\setminus e_1$. As before, let $f=(e')^{-1}$.  If we are
      not in Case 2, then one of $e,f$ has not yet been traversed.
    \end{fullwidth}
  
  \item[Case 3.] This is the first arrival to $v$. In this case the weights
    of all edges incident to $v$ are still $a$, except for $f$ which has
    weight $1+a$. Thus the probability of exiting along $f$ is
    \[
    \frac{1+a}{1+d_v a} \geq \frac{1+a}{1+Ka}.
    \]
    where as usual $d_v$ is the degree of the vertex $v$.
    Thus we can couple the LRRW step so that if $Y'_0=1$, then the walk
    exits along $f$ (and occasionally also if $Y'_0=0$).

  \item[Case 4.] Later visits to $v$ when $Y'_0=0$.  Suppose this is the
    $n^{\textrm{th}}$ visit to $v$.  The weight of $f$ is at least $1+a$,
    since we first entered $v$ through $f^{-1}$.  The total weight of edges
    at $v$ is $2n-1+d_v a$. Thus the probability of the LRRW exiting
    through $f$ is
    \[
    \frac{N_t(f)+a}{2n-1+d_v a} \geq \frac{1+a}{2n-1+Ka}
    \]
    where $t$ is the time of this visit to $v$, and $N_t(f)$ is, as usual,
    the number of crossings of the edge $f$ up to time $t$.  Thus we can
    couple the LRRW step so that if $Y'_{n-1}=1$, then the walk exits along
    $f$ (and occasionally also if $Y'_{n-1}=0$).

  \item[Case 5.] Later visits to $v$ with $Y'_0=1$.  Here, $f$ was
    traversed on the first visit to $v$.  Since we are not in Case 2, $e$
    has not yet been traversed.  Since we are not in Case 1, neither has
    $e^{-1}$, and so $e$ still has weight $a$.  In this case, we first
    decide with appropriate probabilities, and independent of the $Y$,
    whether to use one of $\{e,f\}$, or another edge from $v$.  If we
    decide to use another edge, we ignore the $Y$ variables.  If we decide
    to use one of $\{e,f\}$, and this is the $n^\textrm{th}$ time this
    occurs, then $N_t(f)\geq n$ (since we had only chosen $f$ so far in
    these cases, and also used $f^{-1}$ at least once). Thus the
    probability that we select $e$ from $\{e,f\}$ is
    \[
    \frac{a}{N_t(f)+2a} \leq \frac{a}{n+2a}.
    \]
    Thus we can couple the LRRW step so that if $Y_{n-1}=0$, then we select
    $f$.
  \end{description}

  \paragraph*{Domination of $Q(e)$.}
  We now check that on $\DD_\gamma$ we have $Q(e)\leq \bar{Q}(e)$.  Assume
  $\DD_\gamma$ occurs.  When $Y'_0=0$ the coupling considers only the $Y'$
  variables, one at each visit to $v$ until $f$ is used. If $n$ is minimal
  s.t.\ $Y'_n=1$ then $f$ is used on the $(n+1)^{\textrm{st}}$ visit to $v$ or
  earlier. Thus
  \[
  Q(e)\leq M_e \leq n = \bar{Q}(e).
  \]
  Note that if $Y'_0=0$ it is possible that the walk uses $f$ before $e$,
  and then $Q(e)\leq 1$.
  
  If $Y'_0=1$ then the coupling considers only the $Y$s, one at each time
  that either $e$ or $f$ is traversed (and not at every visit to $v$).  If
  $n$ is the minimal $n\geq 1$ s.t.\ $Y_n=1$, then $f$ is used at least $n$
  times before $e$. Thus
  \[
  Q(e)= \frac{1}{M_f} \leq \frac{1}{n} = \bar{Q}(e),
  \]
  and in both cases $Q(e)\leq \bar{Q}(e)$, completing the proof.
\end{proof}

Let us remark briefly on how this argument changes if the weights $a$ are
not equal, or possibly random. In order to get the domination
$Q(e)\leq\bar{Q}(e)$ the variables $Y_j$ and $Y'_j$ are defined
differently. Setting $a_v = \sum_{e\ni v} a_e$, we use
\begin{align*}
  Y_j &= \Bern\left(\frac{a_e}{j+1+a_e+a_f}\right),  &
  Y'_j &= \Bern\left(\frac{1+a_f}{2j+1+a_v}\right).
\end{align*}
This changes the moment estimation, and instead of $C a$ we get
$C(a_v+a_v^{1+s})$. If the $a_e$s are all sufficiently small there is no
further difficulty. If the $a$'s are random, this introduces dependencies
between edges, and some higher moments must be controlled.

\begin{corollary*}
  \cref{T:ExpDecay} holds on graphs where the LRRW is recurrent.
\end{corollary*}

\begin{proof}
  This is just an aggregation of the results of this section:
  \begin{align*}
    \E\left(\frac{W_x}{W_{e_1}}\right)^s
    &= \sum_\gamma\E\left[\left(\frac{W_x}{W_{e_1}}\right)^s {\mathbf 1}
        \{\DD_\gamma\}\right]\\
    \text{by \eqref{E:CS1}} \qquad
    & \leq \sum_\gamma
    \E\left[\prod_{e \in \gamma\setminus
        e_1}\left(\frac{R(e)}{Q(e)}\right)^{2s} \mathbf 1
      \{\DD_{\gamma}\}\right]^{\frac12}  
    \E\left[\prod_{e \in \gamma\setminus e_1} Q(e)^{2s} \mathbf 1
      \{\DD_{\gamma}\}\right]^{\frac12}\\ 
    \textrm{By \cref{L:RUB,L:AUB}}\qquad
    &\le \sum_\gamma
    \left[C(2s)^{|\gamma|-1}\right]^{\frac12}
    \left[\left(C(2s,K)a\right)^{|\gamma|-1}\right]^{\frac12}\\
    &= \sum_\gamma \left(C_0 \sqrt{a} \right)^{|\gamma|-1},
  \end{align*}
  for $C_0=\sqrt{C(2s) C(2s,K)}$. Now, the number of paths of length $\ell$
  is at most $K^\ell$ and the shortest path to $e$ has length
  $\dist(e,v_0)+1$. If we take $a_0$ such that $K C_0 \sqrt{a_0} = \frac12$
  then for $a<a_0$ longer paths give at most a factor of $2$ and so
  \[
  \E\left(\frac{W_x}{W_{e_1}}\right)^s
  \leq 2K (K C_0\sqrt{a})^{\dist(e,v_0)}      \qedhere
  \]
\end{proof}

\section{Recurrence on Bounded Degree Graphs for Small Values of \texorpdfstring{$a$}{a}}
\label{sec:recurr} 

We are now ready to prove \cref{T:recurrence}. As noted above, the main
idea is to approximate the LRRW on $G$ by LRRW on finite balls in $G$.  Let
$R>0$ and let $X^{(R)}$ be the LRRW in the finite volume ball $B_R(v_0)$.  By
\cref{T:DF}, for each $R$ the full distribution of $X^{(R)}$ is given by a
mixture of random conductance models with edge weights denoted by
$(W^{(R)}_e)_{e \in B_R(v_0)}$, and mixing measure denoted by $\mu^{(R)}$.
Recall $\Omega = \R_+^{E}$ is the configuration space of edge weights on
the entire graph $G$.  For fixed $a$, the measures $\mu^{(R)}$ form a
sequence of measures on $\Omega$, equipped with the product Borel
$\sigma$-algebra.

\begin{proof}[Proof of \cref{T:recurrence}]
  \Cref{T:ExpDecay} (or \ref{T:DF} if you prefer) implies that the measures $\mu^{(R)}$ are tight, and so
  have subsequential limits. Let $\mu$ be one such subsequence limit.
  Clearly the law of the random walk in environment $\bW$ is continuous in
  the weights. However, the first $R$ steps of the LRRW on $B_R(v_0)$ have
  the same law as the first $R$ steps of the LRRW on $G$. Thus the LRRW on
  $G$ is the mixture of random walks on weighted graphs with the mixing
  measure being $\mu$. 

  Fix some $s<\nicefrac14$ for the remainder of the proof. For any edge $e\in
  B_R(v_0)$ we have by Markov's inequality that
  \[
  \mu^{(R)}\left(W^{(R)}_e > Q \right) \leq \frac{2K
    (C\surd{a})^{\dist(e,v_0)}} {Q^s}, 
  \]
  where $C=C(s,K)$ comes from \cref{T:ExpDecay}, which we already proved
  for recurrent graphs. Taking $Q=(2K)^{-\dist(e,v_0)}$, and since the
  number of edges at distance $\ell$ from $v_0$ is at most $K^{\ell+1}$, we
  find that the probability of having an edge at distance $\ell$ with $W_e
  > (2K)^{-\ell}$ is at most
  \[
  K^{\ell+1} (2K)^{s \ell} 2K (C\surd{a})^\ell = 2K^2 \left(2^{s} K^{1+s} C
    \surd a \right)^\ell.
  \]
  Let $a_0$ be such that $2^s K^{1+s} C \sqrt{a_0} = \frac12$. Then for
  $a<a_0$ this last probability is at most $K^2 2^{1-\ell}$. Crucially,
  this bound holds uniformly for all $\mu^{(R)}$, and hence also for $\mu$.
  By Borel-Cantelli, it follows that $\mu$-a.s., for all but a finite
  number of edges we have $W_e \leq (2K)^{-\dist(e,v_0)}$. On this event
  the random network $(G,\mathbf W)$ has finite total edge weight $W=\sum_e
  W_e$, and therefore the random walk on the network is positive recurrent,
  with stationary measure $\pi(v) = W_v/2W$.
\end{proof}

\subsection{Trapping and Return Times on Graphs}

With the main recurrence results, \cref{T:recurrence,T:ExpDecay} proved, we
wish to make a remark about the notions of ``localization'' and
``exponential decay''. We would like to point out that one should be
careful when translating them from disordered quantum systems into the
language of hyperbolic non-linear $\sigma$-models and the LRRW. To
illustrate this point, let $G$ be a graph with degrees bounded by $K$, and
consider the behavior of the tail of the return time to the initial vertex
$v_0$. Let
\[
\tau = \min\{t>0: X_t=v_0\}.
\]
\begin{proposition}
  Let $K$ be fixed.  Suppose that $G$ is a connected graph with degree
  bound $K$ and some edge not incident to $v_0$.  Then there is a constant
  $C>0$ depending on $K$ but not on $a$ so that
  \[
  \P(\tau > M) \geq c(a,K)M^{-(K-1)a}.
  \]
\end{proposition}

Thus despite the exponential decay of the weights, the return time has at
most $(K-1)a-\frac12$ finite moments.

\begin{proof}
  There is some edge $e$ at distance $1$ from $v_0$. Consider the event $E_M$
  that the LRRW moves towards $e$ and then crosses $e$ back and forth $2M$
  times. On this event, $\tau>2M$. The probability of this event is at
  least
  \begin{align*}
    \P(E_M)
    &\geq \frac{1}{K}\left[\prod_{0\leq j<M}\frac{2j+a}{2j+1+Ka} \right]\left[\prod_{0\leq j<M} \frac{2j+1+a}{2j+1+Ka} \right]\\
    &=\frac{1}{K}\prod_{0\le j<M}\left(1-\frac{(K-1)a}{2j+1}+O\Big(\frac{1}{j^2}\Big)\right)^2\\
    &\geq c(a,K) M^{-(K-1)a}.   \qedhere
\end{align*}
\end{proof}

One might claim that we only showed this for the first return time. But in
fact, moving to the RWRE point of view shows that this is a phenomenon that
can occur at any time. Indeed, if $\P(\tau>M)\ge\eps$ then this means that
with probability $\ge \frac 12\eps$, the environment $\bW$ satisfies that
$\P(\tau>M\,|\,\bW)\ge\frac 12\eps$. But of course, once one conditions on
the environment, one has stationarity, so the probability that the
$k^\textrm{th}$ excursion length (call it $\tau_k$) is bigger than $M$ is
also $\ge\frac 12 \eps$. Returning to unconditioned results give that
\[
\P(\tau_k>M)\ge\frac 14 \eps^2=c(a,K)M^{-2(K-1)a} \qquad\forall k
\]
This is quite crude, of course, but shows that the phenomenon of polynomial
decay of the return times is preserved for all times.

\subsection{Reminiscences}

Let us remarks on how we got the proof of these results. We started with a very simple heuristic picture, explained to us by Tom Spencer: when the initial weights are very small, the process gets stuck on the first edge for very long. A simple calculation shows that this does not hold forever but that at some point it breaks to a new edge. it then runs over these two edges, roughly preserving the weight ratio, until breaking into a third edge, roughly at uniform and so on. 

Thus our initial approach to the problem was to try and show \emph{using only the LRRW picture} that weights stabilize. This is really the point about the linearity --- only linear reinforcement causes the ratios of weights of different edges to converge to some value. One direction that failed was to show that it is really independent of the rest of the graph. We tried to prove that ``whatever the rest of the graph does, the ratio of the weights of two edges is unlikely to move far, when both have already been visited enough times''. So we emulated the rest of the graph by an adversary (which, essentially, decides through which edge to return to the common vertex), and tried to prove that any adversary cannot change the weight ratio. This, however, is not true. A simple calculation that shows that an adversary which simply always returns the walker through edge $e$ will beat any initial advantage the other edge, $f$ might have had.

Trying to somehow move away from a general adversary, we had the idea that the RWRE picture allows to restrict the abilities of the adversary, which finally led to the proof above (which has no adversary, of course). We still think it would be interesting to construct a pure LRRW proof, as it might be relevant to phase transitions for other self-interacting random walks, which are currently wide open.

\section{Transience on Non-amenable Graphs}\label{sec:nonamen}

In this section we prove \cref{T:transience}. The proof is ideologically similar to that of \cref{T:ExpDecay}, and in particular we prove the main lemmas under the assumption that the process is \emph{recurrent} --- there it seemed natural, here it might have seemed strange, but after reading \cref{sec:recurr} maybe not so much: in the end we will apply these lemmas for finite subgraphs of our graph. We will then use compactness, i.e. Theorem \ref{T:DF}, to find a subsequential limit on $G$ which also describes the law of LRRW on $G$.

Let us therefore begin with these lemmas, which lead to a kind of "local
stability" of the environment $\bW$ when the initial weights $\mathbf a$
are uniformly large.

As in the case of small $a$, our stability argument has two parts. We use the dynamic
description to argue that the walk typically uses all edges leaving a
vertex roughly the same number of times (assuming it enters the vertex
sufficiently many times). We then use the random conductance view point to
argue that the random weights are typically close to each other. Finally,
we use a percolation argument based on the geometry of the graph to deduce
a.s.\ transience.

\subsection{Local Stability if LRRW is Recurrent}
Let us assume throughout this subsection that the graph $G$ and the weights $a$ are given and that they satisfy that LRRW is recurrent.

Let $L$ be some parameter which we will fix later. The main difference between the proofs here and in \cref{sec:proof} is that here we will examine the process until it leaves a given vertex $L$ times through each edge rather than once. Let therefore $\tau = \tau(L,v)$ denote the number of visits to
$v$ until the LRRW has exited $v$ at least $L$ times along each edge $e \ni
v$.  Note that since we assume that the LRRW is recurrent, these stopping
times are a.s.\ finite.  Let $M(e) = M(L,v,e)$ denote the total number of
{\em outgoing} crossings from $v$ along $e$ up to and including the
$\tau^\textrm{th}$ exit.  Given the environment $\bW$, call a vertex $v$ {\bf
  $\eps$-faithful} if $M(e)/\tau$ is close to its asymptotic value of
$W_e/W_v$, i.e.\
\[
\frac{M(e) / \tau}{W_e/W_v} \in [1-\eps, 1+\eps] \qquad \text{ for all
  $e\ni v$}.
\]

\begin{lemma}\label{L:faithful}
  Let the degree bound $K$ be fixed.  Then for any $\eps,\delta>0$ there
  exists $L = L(K,\eps,\delta)$ so that 
  \[
  \P^{\bW} (v \text{ is not $\eps$-faithful}) < \delta.
  \]
  Moreover, these events are independent under $\P^{\bW}$.
\end{lemma}

A crucial aspect of this lemma is that $L$ does not depend on the
environment $\bW$.  A consequence is that for this $L$, the $\eps$-faithful
vertices stochastically dominate $1-\delta$ site percolation on $G$ for any
$\P^{\bW}$, and hence also for the mixture $\P$.

\begin{proof}
  Given $\bW$, the exits from $v$ are a sequence of i.i.d.\ random
  variables taking $d_v\leq K$ different values with some probabilities
  $p(e)=W_e/W_v$.  Let $A_n(e)$ be the location of the $n^\textrm{th}$ appearance of
  $e$ in this sequence, so that $A_n$ is the cumulative sum of $\Geom(p)$
  random variables. By the strong law of large numbers, there is some $L$
  so that with arbitrarily high probability (say,\ $1-\delta$), we have
  \begin{equation}\label{eq:LLN}
    \left|\frac{pA_n}{n} - 1\right| < \eps \qquad \text{for all } n\geq L.
  \end{equation}

  We now claim that $L$ can be chosen uniformly in $p$ (for $\eps$ and $\delta$ fixed). This can be proved either by a replacing the law of large numbers by a quantitative analog, say a second moment estimate, or by using continuity in $p$: Indeed, the only possible discontinuity is at zero, and $pA_n$ converges as $p\to0$ to a sum of i.i.d.\ exponentials,
  and so with this interpretation the same claim holds also for
  $p=0$. Since the probability of a large deviation at some large time for
  a sum of i.i.d.\ random variables is upper semi-continuous in the
  distribution of the steps, and since upper semi-continuous functions on
  $[0,1]$ are bounded from above, we get the necessary uniform bound.

  We now move back from the language of $A_n$ to counting exits at specific
  time $t$. Denote $M_t(e)$ the number of exits from $v$ through $e$ in the
  first $t$ visits to $v$.  If $t\in [A_n,A_{n+1})$ then $M_t=n$ and
  hence, for $t\ge A_L$ and with probability $>1-\delta$, 
\[
\frac{M_t(e)}{pt}\in\left(\frac{n}{pA_{n+1}},\frac{n}{pA_n}\right]\stackrel{\textrm{\cref{eq:LLN}}}{\subset}[1-\eps',1+\eps']
\]
(where $\eps'$ is some function of $\eps$ with $\eps'\to 0$ when $\eps\to 0$).

Since $\tau\ge A_L(e)$ for all edges $e$ coming out of $v$, we get
$M(e)/p\tau \in [1-\eps',1+\eps']$ for all $e$, with probability
$>1-K\delta$. Replacing $\eps'$ with $\eps$ and $K\delta$ with $\delta$
proves the lemma.
\end{proof}

The second step is to use the dynamical description of the LRRW to show
that if $a$ is large then the $M(L,v,e)$ are likely to be roughly
equal. Let $S(L,v) = \max_{e\ni v} M(L,v,e)$ be the number of exits
from $v$ along the most commonly used edge. By the definition,
$M(L,v,e) \geq L$ for all $e, v$.  We will therefore call a vertex {\bf $\eps$-balanced} if $S(L,v) \leq (1+\eps)L$.

\begin{lemma}\label{L:balanced}
  For any $K$, $\eps$, $\delta$ there is an $L_0=L_0(K,\eps,\delta)$ such that for any $L>L_0$ there is some
  $a_0$ so that for any $a>a_0$ such that LRRW on $G$ with initial weight $a$ is recurrent we have
  \[
  \P(v \text{ is not $\eps$-balanced}) < \delta.
  \]
  Moreover, for such $a$, the $\eps$-balanced vertices stochastically
  dominate $1-\delta$ site percolation.
\end{lemma}

\begin{proof}
  We prove directly the stochastic domination. To this end, we prove that
  any $v$ is likely to be $\eps$-balanced even if we condition on arbitrary
  events occurring at other vertices.  At the $i^\textrm{th}$ visit to $v$, the
  probability of exiting via any edge $e$ is at least $a/(d_v
    a+2i-1)$. Throughout the first $T$
  visits to $v$, this is at least $1/d_v - 2T/a$. Since
  this bound is uniform in the trajectory of the walk, we find that the
  number of exits along an edge $e$ stochastically dominates a
  $\Bin(T,\frac{1}{d_v} - \frac{2T}{a})$ random variable, even if we
  condition on any event outside of $v$.
  
  We take $T = (1+\eps)d_v L$. If $a$ is large enough ($C K^2 L/\eps$
  suffices) then the binomial has expectation at least $L-\frac 12 \eps L$.
  Since it has variance at most $T=O(L)$, if $L$ is sufficiently large then
  the binomial is very likely to be at least $L$.  In summary, given
  $\delta$ and $\eps$ we can find some large $L$ so that for any large
  enough $a$, with probability at least $1-\delta$ there are at least $L$
  exits along an any edge $e$ up to time $T$. This occurs for all edges
  $e\ni v$ with probability at least $1-K\delta$.

  Finally notice that if all edges are exited at least $L$ times, then this
  accounts for $d_vL$ exits, and only $T-d_vL$ exits remain unaccounted
  for. Even if they all happen at the same edge, that edge would still have
  only $L+\eps d_VL$ exits, hence $S\le L+\eps d_vL$. Therefore $v$ is
  $\eps d_v$-balanced with probability at least $1-K\delta$. Redefining
  $\eps$ and $\delta$ gives the lemma.
\end{proof}

Call a vertex {\bf $\eps$-good} (or just good) if it is both
$\eps$-faithful and $\eps$-balanced. Otherwise, call the vertex {\bf bad}.
Note that if $v$ is good, then weights of edges leaving $v$ differ by a factor
of at most $\frac{1+\eps}{1-\eps} \leq (1-\eps)^{-2}$.

\begin{corollary}
\label{C:Main}
  For any $K,\eps,\delta$, for any large enough $L$, for any large enough
  $a$, the set of $\eps$-good vertices stochastically dominates the
  intersection of two Bernoulli site percolation configurations with
  parameter $1-\delta$.
\end{corollary}

Unfortunately, these two percolation processes are not independent, so we
cannot merge them into one $(1-\delta)^2$ percolation.

\subsection{Application to Infinite Graphs}

Let $G_n$ denote the ball of radius $n$ in $G=(V, E)$ with initial vertex
$v_0$. Let $\mu_n$ denote the mixing measure guaranteed by the first half
of Theorem \ref{T:DF}. Further let $\mathcal P_n$ be the sequence of
coupling measures guaranteed by Corollary \ref{C:Main}.

According to the second half of Theorem \ref{T:DF} the measures $\mu_n$ are
tight.  Quite obviously the remaining marginals of $\mathcal P_n$ are also
tight.  Therefore $\mathcal P_n$ is tight.  Thus we can always pass to a
subsequential limit $\mathcal P$ so that the first marginal is a mixing
measure for $\mu$ and the conclusion of Corollary \ref{C:Main} holds.  We
record this in a proposition:

\begin{proposition}
  For any $K,\eps,\delta$, for any large enough $L$, for any large enough
  $a$ the following holds. For any weak limit $\mu$ of finite volume mixing
  measures there is a coupling so that the set of $\eps$-good vertices
  (with respect to $\mu$) stochastically dominates the intersection of two
  Bernoulli site percolation configurations with parameter $1-\delta$.
\end{proposition}

We now use a Peierls' argument to deduce transience for large enough $a$.
We shall use two results concerning non-amenable graphs. The standard
literature uses edge boundaries so let us give the necessary definitions:
we define the edge boundary of a set by $\partial_\textrm{E}(A)=\{(x,y)\in E
:x\in A, y\notin A\}$, $\Vol A = \sum_{v \in {A}} d_v$ and 
\[
\iota_\textrm{E}=\inf_{A\subset G, \Vol A<\infty}\frac{|\partial_\textrm{E}(A)|}{\Vol A}.
\]
Clearly
\[
\iota =\inf\frac{|\partial A|}{|A|}\ge\inf\frac{|\partial_\textrm{E}A|/K}{K\Vol A}=\frac 1{K^2} \iota_\textrm{E}.
\]
and similarly $\iota \le K^2\iota_\textrm{E}$.

The first result that we will use is Cheeger's inequality:

\begin{theorem}\label{T:kesten}
  If $G$ is non-amenable then the random walk on $G$ has return
  probabilities $p_n(0,0) \leq C e^{-\beta n}$, where $\beta>0$ depends
  only on the Cheeger constant $\iota_\textrm{E}(G)$.
\end{theorem}

Cheeger proved this for manifolds. See e.g.\ \cite{D84} for a proof in the
case of graphs. A second result we use is due to Vir\'ag \cite[Proposition
3.3]{Virag}. Recall that the anchored expansion constant is defined as
\[
\alpha(G) = \lim_{n \rightarrow \infty} \inf_{|A| \geq n}
\frac{|\partial_\textrm{E} A|}{\text{Vol}(A)} 
\]
where the infimum ranges over {\em connected} sets containing a fixed
vertex $v$. It is easy to see that $\alpha$ is independent of the choice of
$v$.

\begin{proposition}[\cite{Virag}]\label{P:anchored}
  Every infinite graph $G$ with $\alpha$-anchored expansion contains for
  any $\eps>0$ an infinite subgraph with edge Cheeger constant at least
  $\alpha-\eps$.
\end{proposition}

\begin{proof}[Proof of \cref{T:transience}]
  Let $\eps$ be some small number to be determined later, and let $G_\eps$
  denote the set of $\eps$-good vertices (and also the induced subgraph).
  For any set $F$ with boundary $\partial F$, the $\eps$-bad vertices in
  $\partial F$ are a union of two Bernoulli percolations, each of which is
  exponentially unlikely to have size greater than than $\frac14|\partial
  F|$, provided that $\delta<\frac14$. Specifically,
  \[
  \P\left(|G_\eps \cap \partial F| \leq \frac12 |\partial F| \right) \leq
  e^{-c|\partial F|} \leq e^{-c \iota |F|},
  \]
  where $c=c(\delta)$ tends to $\infty$ as $\delta\to 0$, and $\iota>0$ is
  the Cheeger constant of the graph. The number of connected sets $F$ of
  size $n$ containing the fixed origin $o$, is at most $2^{Kn}$ (this is
  true in any graph with degree bound $K$ --- the maximum is easily seen to
  happen on a $K$-regular tree, on which the set can be identified by its
  boundary which is just $Kn$ choices of whether to go up or down). Thus if
  $\delta$ is small, so that $c(\delta) > K(\log 2)/\iota$, then a.s.\ only
  finitely many such sets $F$ have bad vertices for half their boundary.

  Taking such a $\delta$, we find that $G_\eps$ contains an infinite cluster
  $\CC$ which has ``vertex anchored expansion'' at least $\iota(G)/2$.
  Moving to edge anchored expansion loses a $K^2$ and we get $\alpha\ge
  \iota/2K^2$. By \cref{P:anchored} we find that $\CC$ contains a subgraph
  $H$ with $\iota_\textrm{E}(H)>\iota(G)/3K^2$.

  By \cref{T:kesten}, the simple random walk on $H$ has exponentially
  decaying return probabilities. However, since all vertices of $H$ are
  $\eps$-good, edges incident on vertices of $H$ have weights within
  $(1-\eps)^{-2}$ of equal.  Thus the random walk on the weighted graph
  $(H,\bW)$ is close to the simple random walk on $H$.  Specifically,
  letting $p^\bW$ denote the heat kernel for the $\bW$-weighted random walk
  restricted to $H$, then we find
  \[
  p^\bW_n(0,0) \leq (1-\eps)^{-2n} p_n(0,0) \leq C (1-\eps)^{-2n} e^{-\beta
    n},
  \]
  where $\beta$ is some constant depending only on $G$.  In particular for
  $\eps$ small enough, these return probabilities are summable and the walk
  is transient.

  Finally by Rayleigh monotonicity (see e.g.\ \cite{Doyle}), $(G,\bW)$ is
  transient as well, completing the proof.
\end{proof}

Since there are several parameters that need to be set in the proof, let us
summarize the final order in which they must be chosen.  The graph $G$
determines $K$ and $\iota(G)$.  Then $\delta$ is chosen to get large enough
anchored expansion. This determines via \cref{T:kesten} the value of
$\beta$ for the subgraph $H$, which determines how small $\eps$ needs to
be.  Finally, given $\eps$ and $\delta$ we take $L$ large enough to satisfy
\cref{L:faithful,L:balanced} and the minimal $a$ is determined from
\cref{L:balanced}.

\section{Vertex Reinforced Jump Process}
\label{S:VRJP}

In this section we apply our basic methods to the VRJP models mentioned in
the introduction. This demonstrates the flexibility of the approach and
gives a second proof of recurrence of LRRW based on the embedding of LRRW
in VRJP with initial rates $\bJ$ i.i.d.\ with marginal distribution
$\Gamma(a, 1)$.

\subsection{Times they are a changin'}

As explained in the introduction, the VRJP has a dynamic and an RWRE
descriptions, related by a time change. Let us give the details. The
dynamic version we will denote by $Y_t$, and its local time at a vertex $x$
by $L_x(t)$, so that $t = \sum_x L_x(t)$. Recall that $Y_t$ moves from $x$
to $y$ with rate $J_{x,y}(1+L_y(t))$.

The RWRE picture is defined in terms of a positive function $\bW =
(W_v)_{v\in V}$ on the \textit{vertices}. Given $\bW$ we will denote by
$Z_s=Z_s^\bW$ the random walk in continuous time that jumps from $x$ to $y$
with rate $\frac 12 J_{xy}W_y/W_x$. We will denote the local time of $Z$ by
$M_x(s)$, and again $s=\sum_x M_x(s)$. When some random choice of $\bW$ is
clear from the context we will denote by $Z_s$ the mixed process.

The time change relating $s$ and $t$ is then given by the relation that
$M_x = L_x^2+2L_x$, or equivalently that $L_x = \sqrt{1+M_x} - 1$. Summing
over all vertices gives a relation between $s$ and $t$. Since the local
times are only increasing at the presently occupied vertex, this gives the
equivalent relations $ds = 2(1+L_{Y_t}(t)) dt$ and $dt =
ds/2\sqrt{1+M_{Z_s}(s)}$.

\begin{theorem}[\cite{ST}]\label{T:ST}
  On any finite graph there exists a random environment $\bW$ so that
  $(Z_s)$ is the time change of $(Y_t)$ given above.
\end{theorem}

For the convenience of the reader, here is a little table comparing our
notation with those of \cite{ST}:

\begin{center}
\begin{tabular}{|l||c|c|c|}
\hline 
Here & $J$ & $W$ & $L$ \\
\hline
\cite{ST} & $W$ & \rule{0pt}{12pt}$e^U$ & $L-1$ \\
\hline
\end{tabular}
\end{center}

\noindent In fact, Sabot and Tarr\`es also give an explicit formula for the
law of the environment $\bW$. However, as with the LRRW, we do not require
the formula for this law, but only that it exists.

One more way to think of all this is that the process has two clocks
measuring the occupation time at each vertex. From the VRJP viewpoint, the
process jumps from $i$ to $j$ at rate $J_{ij}(1+L_j(t)) dt$ (i.e.\ with
respect to the $L$'s). From the RWRE viewpoint, the process jumps at rate
$\frac12 J_{ij} W_j/W_i ds$. \cref{T:ST} states that the above
relation between $ds$ and $dt$, these two descriptions give the same law
for the trajectories. 

It is interesting to also describe the process in terms of both the local
times and reinforcement. Using the time of the $L_x$s, and given the
environment, we find that the process jumps from $i$ to $j$ at rate $J_{ij}
(W_j/W_i)(1+L_i) dt$. Note that here it is $L_i$ and not $L_j$ that
controls the jump rates. Similarly, we can get a reinforcement description
for $Z$: the process jumps from $i$ to $j$ at rate $J_{ij}
\sqrt{\rule{0pt}{8pt}\smash{(1+M_j)/(1+M_i)}} ds$. This gives a
description of $Z$ with no random environment, but using reinforcement.
Nevertheless, it will be most convenient to use $Z$ for the RWRE side of
the proof and $Y$ for the dynamical part.

A final observation is that conditional on $\bJ$ and $\bW$, the
discretization of $Z_s$ (i.e.\ the sequence of vertices visited by the
process) has precisely the same law as the random walk on $G$ with
conductances given by $C_{ij}=J_{ij} W_i W_j$.

\subsection{Guessing the environment}

As with the LRRW, the main idea is to extract from the processes some
estimate for the environment, and show that it is reasonably close to the
actual environment on the one hand, and behaves well on the other.

For neighboring vertices $i,j$, let $S_{ij}$ be the first time at which
$Z$ jumps from $i$ to $j$, and let $\tau_{ij} = M_i(S_{ij})$ be the local
time for $Z$ at $i$ up to that time. Given the environment $\bW$, we have
that $\tau_{ij} \eqd \Exp\left(\frac12 J_{ij} W_j/W_i\right)$. Thus we can use
$Q_{ij} := \sqrt{\tau_{ji}/\tau_{ij}}$ as an estimator for
$R_{ij} = W_j/W_i$. 

For a vertex $v$, we consider a random simple path $\gamma$ from $v_0$ to
$v$, where each vertex $x$ is preceded by the vertex from which $x$ is
entered for the first time. This $\gamma$ is just the backward loop erasure
of the process up to the hitting time of $v$. As for the LRRW, we need two
estimates. First an analogue of \cref{L:RUB}:

\begin{lemma}\label{L:R1UB}
  For any simple path $\gamma$ in a finite graph $G$, any environment $\bW$
  and any $s<1$ we have
  \[
  \E^\bW \prod_{e\in\gamma} \left(\frac{R_e}{Q_e}\right)^{2s} =
  \left(\frac{\pi s}{\sin\pi s}\right)^{|\gamma|}.
  \]
\end{lemma}

Second, we need an analogue of \cref{L:AUB}. Recall from \cref{sec:proof}
that $\mathscr D_{\gamma}$ denotes the event that the backward loop erasure
from $v$ is a given path $\gamma$. We use the same notation here.

\begin{lemma}\label{L:A1UB}
  There exists some $C>0$ such that for any $0<s<\nicefrac14$ the
  following holds.  For any finite graph $G$, any conductances $\bJ$, and
  simple path $\gamma$ starting at $v_0$,
  \[
  \E \prod_{e\in\gamma} Q_e^{2s} \mathbf 1\{\DD_\gamma\}
  \le C(s)^{|\gamma|} \prod_{e\in\gamma} J_e^{2s}.
  \]
\end{lemma}

\begin{proof}[Proof of \cref{L:R1UB}]
  Given the environment, the time $Z$ spends at $i$ before jumping to $j$ is
  $\Exp\left(\frac12J_{ij}W_j/W_i\right)$ which may be written as
  $(2W_i/J_{ij}W_j) X_{ij}$, where $X_{ij} \eqd \Exp(1)$. Crucially,
  given the environment the variables $X_{ij}$ are all independent. For an
  edge $e=(i,j)$ this gives $R_e/Q_e =
  \sqrt{X_{ij}/X_{ji}}$. Therefore
  \[
  \E^\bW \prod_{e\in\gamma} \left(\frac{R_e}{Q_e}\right)^{2s}
  = \Big(\Gamma(1+s)\Gamma(1-s) \Big)^{|\gamma|}
  = \left(\frac{\pi s}{\sin\pi s}\right)^{|\gamma|},
  \]
  by the reflection identity for the $\Gamma$ function.
\end{proof}

\begin{lemma}\label{L:moments}
  Suppose $0<s<\nicefrac14$, and let $J>0$ be fixed. Let $U \eqd \Exp(J)$ and
  conditioned on $U$, let $V \eqd \Exp(J(1+U))$. Then
  \[
  \E \left( \frac{2V+V^2}{2U+U^2} \right)^s \leq \frac{C}{1-4s} J^{2s},
  \]
  where $C$ is some universal constant.
\end{lemma}

\begin{proof}
  We can reparametrize $U=\frac{X}{J}$ and $V=\frac{Y}{X+J}$, where $X$ and
  $Y$ are independent $\Exp(1)$ random variables. In term of $X,Y$ we have
  \[
  \frac{2V+V^2}{2U+U^2}
  = \frac{2J^2 Y(J+X)+J^2Y^2}{X(J+X)^2(2J+X)}
  < \frac{2J^2Y}{X^3} + \frac{J^2Y^2}{X^4}.
  \]
  We now calculate, using $(a+b)^s \leq a^s+b^s$ for $0<s<1$,
  \begin{align*}
    \E \left( \frac{2J^2Y}{X^3} + \frac{J^2Y^2}{X^4} \right)^s
    & \le J^{2s} \left(
      \E\big(2Y X^{-3} \big)^s + \E\big(Y^2 X^{-4} \big)^s \right) \\
    &= J^{2s} \left( 2^s\E Y^s \E X^{-3s} +  \E Y^{2s} \E X^{-4s}
    \right)
    &&\text{by independence} \\
    &\le \frac{C}{1-4s} J^{2s}
    &&\text{since $\E X^a\le\frac{C}{1+a}$.}
  \end{align*}
  (the inequality $\E X^a\le C/(1+a)$ holds for $|a|<1$, the relevant range
  here). 
\end{proof}

\begin{proof}[Proof of \cref{L:A1UB}]
  Consider an edge $e=(i,j)\in\gamma$, and let $T_{ij}$ be the first time
  $t$ at which $Y_t$ jumps from $i$ to $j$ (and similarly define $T_{ji}$).
  On the event $\DD_\gamma$ the process does not visit $j$ before the first
  jump from $i$ to $j$. Thus $L_j(t)=0$ for all $t<T_{ij}$. Hence the jump
  $i\to j$ occurs at rate $J_{ij}$ whenever $Y$ is at $i$, and so $U :=
  L_i(T_{ij})$ has law $\Exp(J_{ij})$. More precisely, the statement about
  the jump rate implies that we can couple the process $Y$ with an
  $\Exp(J_{ij})$ random variable, so that on the event $\DD_\gamma$ it
  equals $U$.

  Let $V = L_j(T_{ji})$ be the time spent at $j$ before the first jump back
  to $i$. Since $L_i(t) \geq U$ from the time we first enter $j$, the rate
  of such jumps is always at least $J_{ij}(1+U)$, we find that $V$ is
  stochastically dominated by a $\Exp(J_{ij}(1+U))$ random variable.

  All statements above concerning rates of jumps along the edge $e$ hold
  (on the event $\DD_\gamma$), uniformly in anything that the process does
  anywhere else. Thus it is possible to construct such exponential random
  variables for every edge $e\in\gamma$, independent of all other edges, so
  that the $U$ equals the first and $V$ is dominated by the second.
  The claim then follows by \cref{L:moments}, since $\tau_{ij}=2U+U^2$ and
  $\tau_{ji} = 2V+V^2$ by their definitions and the time change formulae.
\end{proof}

\subsection{Exponential decay and \texorpdfstring{\cref{T:VRJP}}{Theorem \ref{T:VRJP}}}

Let $G_R$ denote the ball of radius $R$ around $v_0$. Denote by $\mu^{(R)}$
the VRJP measure on $G_R$ and the corresponding expectation by
$\E^{(R)}$. In the proof, $\E_{\bJ}$ denote expectation with respect to $
\bJ$.  \cref{T:VRJP} follows from the following:

\begin{theorem}
  There is a universal constant $c>0$ such that the following holds. Let
  $G$ be a fixed graph with degree bound $K$. Let $\bJ= (J_{e})_{e \in E}$
  be a family of independent initial rates with
  \[
  \E J_e^{1/5} < c K^{-4}.
  \]
  Then (a.s.\ with respect to $\bJ$) the measures $\mu^{(R)}$ are a tight
  family and converge to a limit $\mu$ on $\R_+^E$ so that the VRJP is a
  time change of the process $Z_s$ in the environment given by $\mu$. The
  limit process is positive recurrent, and the stationary measure decays
  exponentially.
\end{theorem}

The moment condition on $J$ is trivially satisfied in the case that all
$J_e$'s are bounded by some sufficiently small $J_0$. The particular
condition comes from specializing to $s=\nicefrac15$, and can be easily
changed by taking other values of $s$ or by using H\"older's inequality in
place of Cauchy-Schwartz in the proof below. The dependence on $K$ may be
similarly improved.

\begin{proof}
  Combining \cref{L:R1UB,L:A1UB} with Cauchy-Schwartz, for any $v$, any radius
  $R>\dist(v_0,v)$ and any path $\gamma:v_0\to v$ in $G_R$ we have (recall
  $W_{v_0}=1$):
  \begin{align*}
    \E^{(R)} W_v^s \mathbf 1\{\DD_\gamma\}
    &\leq
    \bigg(\E^{(R)} \prod_{e\in\gamma} \left(\frac{R_e}{Q_e}\right)^{2s}
    \bigg)^{1/2} 
    \bigg(\E^{(R)} \prod_{e\in\gamma} Q_e^{2s} \mathbf 1\{\DD_\gamma\}
    \bigg)^{1/2} \\
    &\leq 
    C_1^{|\gamma|} \prod_{e\in\gamma} J_e^s
  \end{align*}
  where $C_1$ depends only on $s$.  Let the $c$ from the statement of the
  theorem be $1/C_1(\frac 15)$. Then with $s=\nicefrac15$ we get
  \[
  \E_{\bJ} \E^{(R)} W_v^{1/5} \mathbf 1\{\DD_\gamma\} \leq K^{-4|\gamma|}.
  \]
  Since the number of paths of length $n$ is at most $K^n$ and no path to
  $v$ is shorter than $\dist(v_0,v)$ this implies 
  \[
  \E_{\bJ} \E^{(R)} W_v^{1/5} < 2K^{-3\dist(v_0,v)}.
  \]
  Since this bound is uniform in $R$, the Borel-Cantelli Lemma, applied
  with respect to $\bJ$, implies that the measures $\mu^{(R)}$ are tight and
  that they have subsequential limits $\bJ$. Let $\mu$ be any such
  subsequential limit (we later deduce that $\mu$ is unique). It is easy to
  see that the weak convergence of $\mu^{(R)}$ to $\mu$ implies a convergence
  of $Z^{(R)}$, $Y^{(R)}$ and the time change between them to $Z$, $Y$ and the time
  change between them corresponding to the infinite measure (all
  convergences are along the chosen subsequence). However, from the
  reinforcement viewpoint, $Y_t$ has the same law on all $G_R$ until the
  first time it reaches the boundary of the ball. Thus $\mu$ yields the
  VRJP on the infinite graph $G$.

  As noted above, the discretized $Z_s$ is just a random walk on $G$ with
  conductances $C_{ij} = J_{ij} W_i W_j$. By Markov's inequality, $\P(W_v >
  K^{-3\dist(v_0,v)}) \leq 2K^{-2\dist(v_0,v)}$. By Borel-Cantelli, it
  follows that a.s.\ $W_v\leq K^{-3\dist(v_0,v)}$ for all but finitely many
  $v$, and therefore $C_e \leq J_{ij} K^{-6\dist(e,v_0)}$. However, the
  number of edges at distance $n$ is at most $K^n$ and we assumed $J_e$ has
  a finite $\nicefrac15$ moment so yet another application of the Borel-Cantelli
  Lemma ensures that $\sum_e J_e K^{-6\dist(v_0,v)} < \infty$. Thus $\sum_e
  C_e$ is almost surely finite, and the total weight outside $G_R$ decays
  exponentially. This implies the positive recurrence.

  Finally, since the process is a.s.\ recurrent, $Z$ visits each vertex
  infinitely often, and the environment can be deduced from the observed
  jump frequencies along edges, the subsequential limit $\mu$
  is in fact unique. With tightness, this implies convergence of the $\mu_R$.
\end{proof}

\vfill

\noindent {\sc Omer Angel:}
{\tt angel@math.ubc.ca}\\
\noindent {\sc Nicholas Crawford:}
{\tt nickc@tx.technion.ac.il}\\
\noindent {\sc Gady Kozma:}
{\tt gady.kozma@weizmann.ac.il}

\end{document}